\documentclass[reqno,11pt,tbtags,intlimits]{amsart}
\usepackage{amsmath,amssymb,mathrsfs,amsthm,amsfonts}
\usepackage[inline]{enumitem}
\usepackage[usenames,dvipsnames]{xcolor}
\usepackage{thmtools}
\usepackage[percent]{overpic}
\usepackage{comment}
\usepackage[foot]{amsaddr}
\usepackage{mathtools,bbm}
\usepackage{esint}
\usepackage[paper=letterpaper,margin=1in,marginparwidth=1.5cm]{geometry}
\usepackage{tikz}
\counterwithin{figure}{section}

\usepackage[hypertexnames=false]{hyperref}
\usepackage{cleveref}

\hypersetup{%backref=true,       
    urlcolor= violet,                
    linkcolor=blue,                
    bookmarks=true,                 
    bookmarksopen=false,
    citecolor=olive,
}

\newtheorem{thm}{Theorem}[section]

\newtheorem{lem}[thm]{Lemma}
\newtheorem{cor}[thm]{Corollary}
\crefname{thm}{Theorem}{Theorems}
\crefname{prop}{Proposition}{Propositions}
\crefname{lem}{Lemma}{Lemmas}
\crefname{cor}{Corollary}{Corollaries}

\crefname{claim}{Claim}{Claims}
\crefname{step}{Step}{Steps}

\theoremstyle{definition}

\newtheorem{example}[thm]{Example}
\newtheorem{asss}[thm]{Assumptions}

\crefname{asss}{Assumptions}{Assumptions}
\crefname{ass}{Assumption}{Assumptions}
\crefname{definition}{Definition}{Definitions}
\crefname{example}{Example}{Examples}

\theoremstyle{remark}
\newtheorem{remark}[thm]{Remark}
\crefname{remark}{Remark}{Remarks}

\AddToHook{env/prop/begin}{\crefalias{thm}{prop}}
\AddToHook{env/lem/begin}{\crefalias{thm}{lem}}
\AddToHook{env/cor/begin}{\crefalias{thm}{cor}}
\AddToHook{env/claim/begin}{\crefalias{thm}{claim}}
\AddToHook{env/step/begin}{\crefalias{thm}{step}}
\AddToHook{env/definition/begin}{\crefalias{thm}{definition}}
\AddToHook{env/example/begin}{\crefalias{thm}{example}}
\AddToHook{env/remark/begin}{\crefalias{thm}{remark}}
\AddToHook{env/asss/begin}{\crefalias{thm}{asss}}
\AddToHook{env/ass/begin}{\crefalias{thm}{ass}}

\numberwithin{equation}{section}

\makeatletter 
    \def\@makefnmark{\@textsuperscript{\normalfont\texttt{\@thefnmark}}} 
\makeatother

\newcommand{\norm}[1]{\ensuremath\lVert#1\rVert}

\newcommand{\defeq}{\vcentcolon=}

\renewcommand{\leq}{\leqslant}
\renewcommand{\geq}{\geqslant}

\let\temp\phi
\let\phi\varphi
\let\varphi\temp
\let\temp\varepsilon
\let\varepsilon\epsilon
\let\epsilon\temp

\DeclareMathOperator{\tr}{tr}

\newcommand{\N}{\mathbb{N}}
\newcommand{\Pc}{\mathcal{P}}
\newcommand{\Z}{\mathbb{Z}}
\newcommand{\de}{\partial}
\newcommand{\di}{\mathrm{d}}
\newcommand{\R}{{\mathbb{R}}}

\newcommand{\call}[1]{\ensuremath\mathcal{#1}}
\newcommand{\frk}[1]{\ensuremath\mathfrak{#1}}

\newcommand{\bb}[1]{\ensuremath\mathbb{#1}}

\newcommand{\var}{\mspace{1.8mu}\cdot\mspace{1.8mu}}

\newcommand{\trn}{\dagger}
\newcommand{\bX}{\boldsymbol{X}}

\newcommand{\bs}{\boldsymbol}

\newcommand{\DSet}[1]{\ensuremath[\![#1]\!]}
\newcommand\1{\mathbbm{1}}

\newcommand\mres\llcorner

\begin{document}

\date{\today}

\title[On the unimportance of distant players in sparse network games]{On the unimportance of distant players \\
in sparse stochastic differential network games}
\author[M.\ Cirant]{Marco Cirant\textsuperscript{(1)}}
\address{\textsuperscript{(1)}Dipartimento di Matematica ``T.\ Levi-Civita'' \\ Università degli Studi di Padova \\ Via Trieste 63 \\ 35121 Padova, Italy.}
\author[D.\ F.\ Redaelli]{Davide Francesco Redaelli\textsuperscript{(2)}}
\address{\textsuperscript{(2)}Dipartimento di Matematica \\ Università di Roma Tor Vergata \\ Via della Ricerca Scientifica 1\\ 00133 Roma, Italy.}
\email{cirant@math.unipd.it, redaelli@mat.uniroma2.it}


\subjclass[2010]{}

\begin{abstract} 
We study stochastic differential games with $N$ players, where interactions are determined by sequences of graphs in which the number of neighbours of each node remains bounded as $N$ grows, such as chain graphs or lattices. Our main goal is to quantify the phenomenon of the ``unimportance of distant players'' in such a large population, sparse regime: we show that, in order to determine the optimal trajectory in open-loop strategies of a given player with an arbitrarily small error, it suffices to consider a reduced game involving only the players at a certain distance in the graph, assigning arbitrary trajectories to the farther ones. Our main result provides an explicit non-asymptotic estimate in terms of the graph distance, valid independently of the time horizon $T$, under suitable convexity and monotonicity assumptions on the costs. Similar results are obtained for games in distributed strategies.
\end{abstract}

\maketitle

\setcounter{tocdepth}{2}
%\tableofcontents

\section{Introduction}

We consider an $N$-player dynamics with states $X^{*,i}_t \in \R^d$ governed by the SDEs
\begin{equation} \label{eq_stateop}
\begin{dcases}
\di X^{*,i}_t =  \beta^{*,i}_t \,\di t + \sigma^i\,\di B_t^i & \text{in}\ [0,T] \\
X_{0}^{*,i} \sim m_0^i \in \Pc_2(\R^d)
\end{dcases}
\qquad i \in \DSet N \defeq \{0,\dots,N-1\},
\end{equation}
where $\beta^{*,i}_t$ denotes a \emph{drift} that we are going to specify---see~\eqref{eq_opconOL}, $\sigma^i \in \R^{d \times d_n}$ is the \emph{diffusion matrix} (non-degenerate in the sense that $\Sigma^i \defeq \sigma^i \sigma^i{}^\trn$ is positive definite) and the $B^i$'s are independent $d_n$-dimensional Brownian motions, while $\call P(\R^d)$ denotes the space of probability measures on $\R^d$ with bounded second moment. We consider each drift to be given by
\begin{equation} \label{eq_opconOL}
\beta^{*,i}_t = - D_p H^i(X^{*,i}_t,v^i(t,\bX^*_t))
\end{equation}
for some $C^2$-smooth \emph{Hamiltonian} $H^i = H^i(x,p)$, where $v = (v^i)_{i \in \DSet N} \colon [0,T] \times (\R^d)^N \to (\R^d)^N$ solves the following \emph{Pontryagin PDE system}:
\begin{equation} \label{pspde}
\begin{dcases}
-\de_t v^i - \tfrac12 \sum_{j \in \DSet N} \tr(\Sigma^j D_{jj} v^i) + D_xH^i(x^i,v^i) + \sum_{j \in \DSet N} D_j v^i D_pH^j(x^j,v^j) = D_i f^i \\[-3pt]
v^i(T,\var) = D_ig^i
\end{dcases}
\quad i \in \DSet N,
\end{equation}
where $f = (f^i)_{i\in\DSet N} \colon [0,T] \times (\R^d)^N \to \R^N$ and $g = (g^i)_{i\in\DSet N} \colon (\R^d)^N \to \R^N$ are functions of class $C^2$ with further properties that we will disclose in a moment.

Under standard assumptions, $\bs X^*$ is the vector of trajectories at  \emph{Nash equilibrium in open-loop strategies} $\alpha^i$ for an $N$-player game with dynamics
\begin{equation} \label{eq_state}
\begin{dcases}
\di X^i_t = b^i(X^i_t,\alpha_t^i) \,\di t + \sigma^i\,\di B_t^i & \text{in}\ [0,T] \\
X_0^i \sim m_0^i \in \Pc_2(\R^d)
\end{dcases}
\end{equation}
and \emph{costs}
\begin{equation} \label{eq_gcosts}
J^i(\alpha) = \bb E \biggr[\int_0^T \Bigl( L^i(X^i_t,\alpha^i_t) + f^i(t,\bX_t) \Bigr)\,\di t + g^i(\bX_T) \biggr],
\end{equation}
once the Hamiltonians are defined by
\[
H^i(x,p) = \sup_{a \in \R^d} \Bigl( - b^i(x,a) \cdot p - L^i(x,a) \Bigr), \qquad x,p \in \R^d.
\]
In this case, $v$ is the so-called \emph{decoupling field} for the associated \emph{Pontryagin FBSDE system}. We refer the reader to \cite{cardelar} for more precise information.

%Under the choice \eqref{eq_opconOL}, system~\eqref{eq_stateop} can describe the Nash equilibrium trajectories of an $N$-player game in \emph{open-loop strategies}, as explained in \Cref{rmk_corrtogam}. 
We are particularly interested in the study of \emph{sparse} games, in which interactions between the $N$ players, occurring through the functions $f^i$ and $g^i$, are determined by some sequence of sparse \emph{underlying graphs} $\Gamma_N$ with vertices in $\DSet N$, in the sense that the functions $f^i$ and $g^i$ only depend on $x^i$ and $x^j$ for $j \sim i$---that is, there is a (possibly directed) edge of $\Gamma_N$ connecting $j$ to $i$---and the number of neighbours of each $i \in \DSet N$ is bounded as $N$ grows. This will be formally stated among \Cref{assOL}.

Note that if $\Gamma$ was complete---that is, $j \sim i$ for all $i,j \in \DSet N$ with $j\neq i$---and $f^i$ and $g^i$ depended on the average of $(X^j)_{j \sim i}$ (and in no other way on $i$) then we would be in the setting of \emph{mean field game} (MFG) theory. Introduced independently by Lasry and Lions \cite{LL07} and Huang, Caines and Malham\'e \cite{hcm}, MFG theory provides a successful model for the study of dynamic games with many players that are both \emph{symmetric}---that is, indistinguishable from one another---and individually \emph{negligible} as the number of players grows (each interaction scales as $1/N$ in the sense that $D_j f^i, D_j g^i \lesssim 1/N$ for $j \neq i$); in such framework, one has a decentralised and symmetric \emph{limit model} for infinitely many players, involving the equilibrium between a typical player and a mass of agents which is realised by a feedback function of only the state of the player, is identical for all the players, and can be computed just by observing the distribution of the population of players at the initial time. See, for example, the book \cite{cardelar} or the lecture notes \cite{CPnotes} for a recent account on MFG theory. In particular, the MFG paradigm presumes that we have an sequence of \textit{complete} graphs $(\Gamma^N)_{N \in \N}$ on $N$ vertices, hence of degree $N-1 \to \infty$; on the contrary, here we are interested in addressing the antipodal problem when $\sup_{i \in \DSet N,\, N \in \N} \deg_{\Gamma_N}\!(i) < \infty$.

When the structural assumptions of an $N$-player dynamic game as described above fall outside the MFG paradigm---that is, the symmetry is broken and/or the scaling of the interaction is changed---one enters the broader framework of \emph{network games} (or \emph{graph games}), which raised increasing interest in the recent years. A characterisation of large population limits is still possible through a generalised mean field description based on the notion of \emph{graphon}, under the fundamental assumption that $\inf_{i \in \DSet N} \deg_{\Gamma_N}(i) \xrightarrow{N\to\infty} \infty$ (at a suitable rate), so one refers to the graph structure as \emph{dense}. See, for instance, \cite{Carmona2,BWZ,CainesHuang,DelarueESAIM,LackerSoretLabel} and the references therein.
On the other hand, when the underlying network (equivalently, graph) structure is sparse, as in the case of our interest, less results are available at the present time. Such games seem to be more elusive because of the lack of a well-established limit model, yet attention on them is rising as they seem more suitable than MFGs for modelling (competitive) systems with limited and local interactions. We refer the reader to \cite{CainesIEEE22,CainesHuACC24,CainesHuIEEE24,LackRamWu} and references therein for further discussion and to have a look at the main approaches to large population sparse games, namely \emph{vertexon-graphexon} and \emph{local weak} limits.

The crucial difference with the dense regime is the loss of the aforementioned negligibility assumption, and thus basically the loss of a fruitful ground to have a Law of Large Numbers: if the costs of the players do not depend on cumulative functions of a diverging number of variables, one is led to expect that most of the players should have a small influence on a given one, with no cumulative effect arising from their interactions in the large population limit. In other words, given a player, we expect players that are far from it with respect to the graph distance to be asymptotically \emph{unimportant} in determining the optimal distribution of that player.

This mechanism of ``independency'' between players that are far in the above sense has been first observed in \cite{LackerSoret} by means of correlation estimates, and then in \cite[Chapter~1]{PHDT} under the more analytic point of view that we are going to present in this work; see also the recent preprint \cite{IMPC} for similar observations on this phenomenon. We have also addressed in \cite{CR24} a similar problem, deriving analogous information: for \emph{closed-loop} games on sparse graphs with a chain structure, we have quantified the influence of player $j$ on player $i$ by estimating the derivative $D_j u^i$ of the solution of the corresponding \emph{Nash system}---the analogue of \eqref{pspde} to describe closed-loop equilibria---in order to show that $|D_j u^i| \to 0$ as $d_{\Gamma}(i,j) \to \infty$,\footnote{$d_\Gamma$ denotes the standard distance on the graph $\Gamma$---that is, $d_\Gamma(i,j)$ is defined as the length $\ell$ of the shortest path $j \sim k_{\ell-1} \sim \cdots \sim k_{1} \sim i$.} and similarly for second-order derivatives. The reader can also have a look at \cite{FFI1,FFI2} for further results about linear-quadratic differential games on chain networks. As a byproduct of the estimates in \cite{CR24}, we have also shown that such a decay is fast enough to guarantee the well-posedness of the linear-quadratic Nash system in infinitely many dimension, while a more general setting is addressed in \cite{R_IDNS}, whereas it is well-known that in the MFG framework the so-called \emph{master equation} needs to be used as the natural corresponding object arising in the large population limit, as first explained in \cite{LionsSeminar} and detailed in \cite{CDLL}.

Here we focus on open-loop games governed by general (sparse) graphs, with the goal of depicting the \emph{unimportance of distant players} presented above under a different perspective. Under the hypotheses of MFG theory, one exploits the averaging structure of the costs to reduce the prospective infinitely many equations describing Nash equilibria in the large population limit to the well-known two-equation \emph{mean field system}; here, for the reasons discussed above, we have no hope to be able to exploit any aggregate effect coming into play as the number of player grows. On the contrary, we shall show that, if one wishes to find (with a small error) the optimal trajectory of a given player, a \emph{reduction} of the complexity of the system to be studied is possible by going in the opposite direction---that is, ``cutting out'' of the game distant players and thus considering a new game governed by just a subgraph of the original one. Most importantly, we obtain explicit estimates that quantitatively capture the interplay between the ``costs'' $f$ and $g$ and the ``shape'' of the graph $\Gamma$ in their contribution to the unimportance of distant players.

We also point out that we have chosen to work in a regime where such a reduction is expected to be possible independently of the length of time horizon---that is, of how large $T$ is---so to emphasise that such a phenomenon is intrinsic of the sparse nature of the game and it is not a consequence of any short-time stability of system~\eqref{eq_state}: our assumptions require in particular $f^i$ and $g^i$ to be convex with respect to $x^i$, see Assumption~\ref{assOL}(c). Nevertheless, all our results can be easily extended to more general settings (such as $f^i$ and $g^i$ only semi-convex in $x^i$) by paying the price of $T$-dependent constants, as done in \cite{PHDT}; the message in this case would be that the length of the horizon surely influences the unimportance of distant players, while the sparse structure still intervene in providing bounds that get worse slower as $T$ diverges with respect to those one would obtain for a game on a dense graph.

In a similar spirit to \cite{CR_ENS,CJR}, our method is \emph{non-asymptotic} in nature; we are not going to deal with (nor exploit) possible limit models here, although our results, interpreted with the large population limit in mind, can also provide the basic step towards quantitative convergence estimates as well as well-posedness of system~\eqref{pspde} in infinitely-many dimensions, by following an approach akin to \cite{CR24,R_IDNS}.

Our main result is \Cref{thmOL}, presented in the next section. We will also give two variations of it: the first one (\Cref{thmOLvar}) allows to obtain also the interpretation of the unimportance in terms of the decay of the derivatives of $v$, which in open-loop games plays the role similar to that of $(D_iu^i)_{i\in\DSet N}$ in closed-loop games; the second one (\Cref{thmD}) shows that the strategy of the proof of \Cref{thmOL} is in fact robust enough to extend it to a wider class of games, including those in \emph{distributed} strategies. %Some illustrative examples are presented in \Cref{sec_ex}, while \Cref{sec_proofs} collects the proofs of the above-mentioned theorems, also making use of two general lemmas contained in \Cref{sec_tl}.

\subsection{Acknowledgements}
The authors acknowledge the support of the Italian (EU Next Gen) PRIN project 2022W58BJ5 (\emph{PDEs and optimal control methods in mean field games, population dynamics and multi-agent models}, CUP E53D23005910006) and of Gruppo Nazionale per l’Analisi Matematica, la Probabilità e le loro Applicazioni (GNAMPA) of the Istituto Nazionale di Alta Matematica (INdAM). D.F.R. also supported by Fondazione Cassa di Risparmio di Padova e Rovigo and by the MUR Department of Excellence project \emph{MatMod@TOV} (awarded to the Department of Mathematics, University of Rome Tor Vergata, CUP E83C23000330006).

\section{Set-up and main results}  %\label{sec_OL}

\subsection{Unimportance of distant players in open-loop games}
We consider the dynamics given by \eqref{eq_stateop}--\eqref{eq_opconOL}--\eqref{pspde} as the equilibrium for an $N$-player open-loop game, with the following structural assumptions.

\begin{asss} \label{assOL} 
The following hold, for each $i \in \DSet N$:
\begin{enumerate}[label=(\alph*),leftmargin=*]
\item the Hamiltonian $H^i$ satisfies 
\[
- D_x H^i\bigr|^{(x,p)}_{(\bar x,\bar p)} \cdot (x - \bar x) + D_p H^i\bigr|^{(x,p)}_{(\bar x,\bar p)} \cdot (p - \bar p) \geq \kappa^i \Bigl| D_p H^i\bigr|^{(x,p)}_{(\bar x,\bar p)} \Bigr|^2 \qquad \forall \, (x,p),(\bar x,\bar p) \in \R^d \times \R^d;
\]
\item there exists $\call N_i \subset \{0,\dots,N-1\} \setminus \{i\}$ such that
\[
D_{j} f^i = D_{j} g^i = 0 \quad \text{for} \quad j \notin \call N_i \cup \{i\};
\]
\item there exist $K_f^i,\ell_f^i > 0$ such that
\[
D_i f^i(t,\var)\bigr|^{\bs x}_{\bs y} \cdot (x^i - y^i) \geq K_f^i |x^i - y^i|^2 - \ell_f^i \sum_{j \in \call N_i} | x^j - y^j |^2 \qquad \forall\, \bs x, \bs y \in (\R^d)^N, \, t \in [0,T],
\]
and likewise for $g$, where we used the notation $h\bigr|^{x}_y \defeq h(x) - h(y)$.
\end{enumerate}
\end{asss}

For instance, considering the correspondence with games highlighted in the Introduction, if $b^i(x,a) = a$, then Assumption~\ref{assOL}(a) is equivalent to $L^i$ being jointly convex, $\kappa^i$-strongly in the control variable. Assumption~\ref{assOL}(b) is the one introducing the graph structure, as discussed in the following \Cref{rmk_NtoG}, while Assumption~\ref{assOL}(c) encodes the strong convexity of $f^i$ and $g^i$ with respect to player $i$'s own state as well as their Lipschitz continuity, that in general also carries the information about the scaling of the interactions.

Besides these assumptions, we will also suppose that system~\eqref{pspde} has a solution of class $C^{1,2}$ with bounded derivatives, and we will refer to it as an \emph{admissible solution}. Note that this is compatible with the expected growth of $v$ in view of \Cref{assOL}(a,c).

\begin{remark} \label{rmk_NtoG}
The subsets $\call N_i$ can be thought to determine, or be determined by, an underlying, possibly directed, graph $\Gamma$ where the node $i$ is connected to the node $j$ if and only if $j \in \call N_i$. Following this interpretation, we will denote by $\call N_i^{(k)}$ the set of all nodes $j$ such that the shortest (directed) path connecting $i$ to $j$ has length $k$ and use the notation
\[
j \sim_k i \quad \iff \quad j \in \call N_i^{(k)},
\]
for each $k \geq 0$ (omitting the subscript $k$ if it is $1$---that is, we will simply write $j \sim i$ instead of $j \sim_1 i$). Note that such a definition amounts to letting $\call N_i^{(0)} \defeq \{i\}$ and 
\begin{equation*}% \label{eq_kneigh}
\call N_i^{(k)} \defeq \bigcup_{j \in \call N^{(k-1)}_i} \call N_j \setminus \bigcup_{h<k} \call N^{(h)}_i \qquad \forall \, k \geq 1.
\end{equation*}
We will also call
\[
n_i \defeq \#\call N_i
\]
the \emph{in-degree} of node $i$, and consequently the \emph{out-degree} of node $i$ will be given, in terms of the subsets $\call N_j$, by
\[
\check n_i \defeq \sum_{j \in \DSet N} \1_{\call N_j}(i),
\]
where $\1_{\call S}$ denotes the characteristic function of the set $\call S$.
%The maximum out-degree and the maximum in-degree will be denoted by $n \defeq \sup_{i\in \DSet N} n_i$ and $\check n \defeq \sup_{i \in \DSet N} \check n_i$, respectively.
\end{remark}

We are going to fix a player ``of interest'' (without loss of generality, indexed by $i=0$) and a set of indices $I = I_0 \subset \DSet N$ that contains $0$, in order to compare that player's dynamics to the one given as follows: let $Z^i_t$ be any $\R^d$-valued stochastic processes with bounded second moment and define
\begin{equation*}%\label{eq_redcosts}
\hat f^i(t,\bs{\hat x}) \defeq \bb E\bigl[ f^i(t,\bs{\hat x}, \bs Z_t)\bigr], \qquad \bs{\hat x} \in (\R^d)^{I}, \ i \in I,
\end{equation*}
and likewise for $\hat g^i$, where the the $i$-th coordinate of the vector $(\bs y, \bs Z_t)$ is $y^i$ if $i \in I$ and $Z^i_t$ otherwise; then let
\begin{equation} \label{eq_stateopred}
\begin{dcases}
\di \hat X^{*,i}_t =  \hat\beta^{*,i}_t \,\di t + \sigma^i\,\di B_t^i & \text{in}\ [0,T] \\
\hat X_{0}^{*,i} \sim m_0^i \in \Pc_2(\R^d)
\end{dcases}
\qquad i \in I,
\end{equation}
where
\begin{equation} \label{eq_opconOLred}
\hat \beta^{*,i}_t = - D_p H^i(\hat X^{*,i}_t, \hat v^i(t,\bs{\hat X}{}^*_t)),
\end{equation}
with $\hat v = (\hat v^i)_{i \in I} \colon [0,T] \times (\R^d)^{I} \to (\R^d)^I$ solving
\begin{equation} \label{pspdered}
\begin{dcases}
-\de_t \hat v^i -  \tfrac12 \sum_{j \in I} \tr(\Sigma^j D_{jj} \hat v^i)  + D_xH^i(x^i,\hat v^i) + \sum_{j \in I} D_j \hat v^i D_pH^j(x^j,\hat v^j) = D_i \hat f^i \\[-3pt]
\hat v^i(T,\var) = D_i \hat g^i
\end{dcases}
\qquad i \in I.
\end{equation}

Having in mind the interpretation the link with games noted in the Introduction, the above construction can be rephrased as follows: we want to compare the optimal trajectory of player $0$ in the original $N$-player game with that for a \emph{reduced} version of it, centered at player $0$, with states given by
\begin{equation*} %\label{eq_redsde}
\begin{dcases}
\di \hat X^i_t = b^i(X^i_t,\alpha_t^i) \,\di t + \sigma^i\,\di B_t^i & \text{in}\ [0,T] \\
\hat X_0^i \sim m_0^i \in \Pc_2(\R^d)
\end{dcases}
\qquad i \in I
\end{equation*}
and costs
\[
\hat J^i(\alpha) = \bb E \biggr[\int_0^T \Bigl( L^i(\hat X^i_t,\alpha^i_t) + \hat f^i(t,\bs{\hat X}_t) \Bigr)\,\di t + \hat g^i(\bs{\hat X}_T) \biggr], \qquad i \in I.
\] 
In particular, if $I$ is a neighbourhood of $0$ on the graph induced by the $\call N_i$'s (recall \Cref{rmk_NtoG}), then this is basically a localisation of the $N$-player game around player $0$. In fact, by Assumption~\ref{assOL}(b), $\hat f^i$ coincides with $f^i$ for $\call N_i \subset I$; if instead there exists $j \in \call N_i \setminus I$, we are saying that the cost of player $i$ is built form the cost of the $N$-player game by making an arbitrary guess on the behaviour of player $j$. Also note that actually the relevant processes among the $Z^i_t$'s are only those with $i \in \bigcup_{j \in I} \call N_j \setminus I$.

Our key estimate is the following; in its statement, $\call W_2$ denotes the usual Wasserstein-$2$ distance on $\call P_2(\R^d)$, and $\frk m_2(m^{*,j}_t)$ and $\frk m_2(Z^j_t)$ are the second moments of $X^{*,j}_t \sim m^{*,j}_t$ and $Z^j_t$, respectively.

\begin{thm} \label{thmOL}
Let \Cref{assOL} be in force. Let $v$ and $\hat v$ be admissible solutions to \eqref{pspde} and \eqref{pspdered}, respectively, and let $X^{*,i}_t \sim m^{*,i}_t$ and $\hat X^{*,i}_t \sim \hat m^{*,i}_t$ be the solutions to \eqref{eq_stateop}--\eqref{eq_opconOL} and \eqref{eq_stateopred}--\eqref{eq_opconOLred}, respectively. Suppose that $I$ is the set of indices at distance at most $r$ from $i=0$, that is,
\begin{equation} \label{asssat}
I = \bigcup_{k \in \DSet r} \call N^{(k)}_0.
\end{equation}
for some $r \in \N$. There exist constants $\theta^*\in(0,1)$ and $\gamma^{(r)} \geq 0$ such that, if
\begin{equation} \label{cond_smallness}
\theta \defeq \sup_{i \in I} \Bigl( \frac{\ell_g^i}{\inf_{j \sim i}( \frac{\kappa^j}{8T} + K_g^j)} \vee \frac{\ell_f^i}{\inf_{j \sim i} (\frac{\kappa^j}{8T^2} + K_f^j)} \Bigr) \leq \theta^*,
\end{equation}
then
\begin{equation} \label{mainform}
 \fint_0^T \call W_2(m^{*,0}_\cdot,\hat m^{*,0}_\cdot)^2 \leq C_r \gamma^{(r)} \sum_{j \sim_{r} 0} \sup_{[0,T]} \bigl( \frk m_2(m^{*,j}_\cdot) + \frk m_2(Z^j_\cdot) \bigr) ,
\end{equation}
with $C_r$ depending only on $\kappa^i,K_g^i,K_f^i$, $i \in \call N_0^{(r)} \cup \{0\}$. In particular, if there is $\frk n \geq 1$ such that
\begin{equation} \label{eq_Nhkbdd}
N^h_k \defeq \# \{ j \in \DSet N :\ k \sim j \sim_h 0 \} \leq \frk n \qquad \forall\, h \in \N, \ k \in \DSet N,
\end{equation}
then for each $\bar\gamma \in (0,1)$ one can choose $\theta^*$ depending only on $\frk n$ and $\bar\gamma$, and such that
\begin{equation} \label{eq_gammaexpd}
\gamma^{(r)} \leq \bar\gamma^r.
\end{equation}
%More explicitly, there is a sequence of non-negative numbers $(\gamma_h)_{h \in \N}$ satisfying $\gamma_0 = \theta$ and
%\begin{equation} \label{eq_boundgamma}
%\Bigl( {\displaystyle 1 - \theta \sum_{i = 0}^{h} \sup_{k \sim_i 0} N_k^h \prod_{j=i}^{h-1} \gamma_{j}} \Bigr) \gamma_{h} = {\displaystyle \theta \sup_{k \sim_{h+1} 0} N_k^h} \qquad \forall \, h \in \DSet r \setminus \{0\},
%\end{equation}
%where
%\begin{equation} \label{eq_Nkhdef}
%N_k^h \defeq \# \{ j \in \DSet N :\ k \sim j \sim_h 0 \},
%\end{equation}
%such that
%\begin{equation} \label{eq_boundgammar}
%\gamma^{(r)} = \prod_{h\in\DSet r} \gamma_h.
%\end{equation}
\end{thm}

The bound \eqref{mainform} provides a general \emph{non-asymptotic} estimate of the distance of the optimal trajectories of the player of interest in the initial and reduced games, without any needed assumptions on the shape of the interactions---that is, on the $\call N_i$'s. On the other hand, the relevant constants explicitly depend on $\Gamma$ (precisely, on the $N_k^h$'s defined in \eqref{eq_Nhkbdd}---see the proof of \Cref{lem_recsum}), so \Cref{thmOL} gives insights on the quantification of the interplay between the structure of the costs and shape of the graph in determining the (un)importance of distant players. In particular, the last part of the statement will be crucial for showing the unimportance of distant players (see \Cref{cor_red}), as it tells that when the $N_k^h$'s are uniformly bounded, the constant $\gamma^{(r)}$ exhibits an exponential vanishing rate; for instance, this happens if the out-degrees are uniformly bounded as one has the trivial bound
\begin{equation} \label{eq_trivbN}
\sum_{h \in \N} N_k^h \leq \check n_k.
\end{equation}
%Further considerations are made in \Cref{rmk_ip} and illustrative examples are proposed in \Cref{sec_ex}.

%\begin{remark} \label{rmk_gammaexp}
%The sequence $(\gamma_h)_{h \in \N}$ being non-negative implies an explicit upper bound for $\theta^*$ that depends on the structure of the graph $\Gamma$, namely
%Furthermore, by \eqref{eq_boundgamma}, $\gamma_h \to 0$ as $\theta \to 0$; therefore, by \eqref{eq_boundgammar} there is a threshold $\bar\theta^* \in (0,1)$ such that, if $\theta \leq \bar\theta^*$, then $\gamma^{(r)}$ vanishes at exponential rate as $r \to \infty$. In particular, the trivial bound
%\begin{equation} \label{eq_trivbN}
%\sum_{h \in \N} N_k^h \leq \check n_k,
%\end{equation}
%along with \eqref{eq_boundgamma}, \eqref{eq_boundtheta} and \eqref{eq_boundgammar}, guarantees that $\bar\theta^* \geq \nu/\sup_{k \in \DSet N} \check n_k$, for some $\nu$ independent of $r$.
%\end{remark}

\begin{remark}
The requirement \eqref{cond_smallness} entails a strong form of displacement monotonicity on the long time horizon. Indeed, \eqref{eq_trivbN} and the explicit upper threshold \eqref{eq_boundthetaA} on $\theta^*$ in the proof of \Cref{lem_recsum} implies that it is in general required that $\sup_{i \in I}\check n_i \leq \inf_{i \in I} \frac{K_\bullet^i}{\ell_\bullet^i}$ (with $\bullet \in \{f,g\}$) for \eqref{cond_smallness} to hold for any $T>0$; under such a condition, Assumption~\ref{assOL}(c) in turn implies 
\[
\sum_{i \in \DSet N} D_i f^i(t,\var)\bigr|^{\bs x}_{\bs y} \cdot (x^i - y^i) \geq 0 \qquad \forall\, \bs x, \bs y \in (\R^d)^N, \, t \in [0,T],
\]
that is the displacement monotonicity of $f=(f^i)_{i \in \DSet N}$ (and likewise for $g$).
\end{remark}

The proof of \Cref{thmOL} will be given in \Cref{sec_proofs}; we anticipate that it essentially consists in obtaining a recursive inequality by a synchronous coupling and then applying the technical \Cref{lem_recsum}. An immediate consequence of \Cref{thmOL} is the following, that makes more apparent how this first result relates to the unimportance of distant players; it applies to graphs with possibly countably many vertices.

\begin{cor} \label{cor_red}
Let \Cref{assOL} be in force, and let $I$ be as in \eqref{asssat}. Let $M>0$ be such that $\frk m_2(m^{*,j}_t),\frk m_2(Z^j_t)\leq M$ for all $j \in \DSet N$ and $t \in [0,T]$. Suppose that there is $\frk n \geq 1$ such that one has \eqref{eq_Nhkbdd} as well as $C_r \#\call N_0^{(r)} \leq \frk n^r$ for all $r \geq 1$, where $C_r$ is the constant in \eqref{mainform}. Then, for any $\epsilon > 0$ there exist $\theta^* \in (0,1)$ and $r^* \in \N$ (depending only on $\epsilon$, $M$ and $\frk n$), such that
\begin{equation} \label{eq_corest}
\begin{dcases}
\theta \leq \theta^* \\
r \geq r^*
\end{dcases}
 \quad \implies \quad \fint_0^T \call W_2(m^{*,0}_\cdot,\hat m^{*,0}_\cdot) < \epsilon.
\end{equation}
\end{cor}

\begin{proof}
Choose $\theta^*$ in such a way that \eqref{eq_gammaexpd} holds with $\bar\gamma < (2\frk n)^{-1}$; then it suffices to take $r^* = \log_2(2M/\epsilon)$.
\end{proof}

The above corollary can be read as follows: given a game based on a graph $\Gamma$ with a very large number of players (possibly infinitely many), in order to determine the optimal trajectory of our player $0$ ``of interest'' with an error less than $\epsilon$, it is sufficient to consider a reduced game based on the smallest subgraph of $\Gamma$ that contains $0$ and all nodes at distance at most $r^* = r^*(\epsilon)$ from $0$, where one assigns arbitrary trajectories $Z^j_t$ to the farthest nodes.

Notably, the result is independent of $T$, so it describes a phenomenon that is genuinely due to the structure of the game and that it is not a consequence of the shortness of the horizon.

We also point out that \Cref{cor_red} is meaningful in illustrating the phenomenon of the unimportance of distant players as long as $\call N_0^{(r^*)} \neq \emptyset$, and thus in general for \emph{large} sparse graphs. If the error $\epsilon^*$ we accept in the approximation of the optimal distribution of player $0$ is too small with respect to the structure of the $N$-player game (for example, $N$ is not large enough, or the interactions are too strong, or they spread too fast---which is encoded in the quantity $\theta$), then \eqref{eq_corest} holds with $\epsilon \leq \epsilon^*$ only if $I = \DSet N$ and thus it holds with $\epsilon = 0$ (as the right-hand side of \eqref{mainform} is $0$ in this case). Such an estimate is no longer comparing distributions relative to games of different dimensions, rather it is proving that the optimal distribution of player $0$ in the $N$-player game is unique, as formalised next.

\begin{cor} \label{cor_uniq}
Let $v_1,v_2$ be two solutions to \eqref{pspde} and let $m_1,m_2$ be the laws of the solutions to \eqref{eq_state} with $\alpha^i$ given by \eqref{eq_opconOL} where $v$ is replaced with $v_1,v_2$, respectively. Suppose that \eqref{cond_smallness} holds with $\theta^*$ as in \Cref{thmOL}. Then $m_1 = m_2$ and, if each $H^i$ is strongly convex in $p$, $v_1 = v_2$, $m_1$-a.e.
\end{cor}

\begin{proof}
For each $i \in \DSet N$, apply \Cref{thmOL} with $i$ in place of $0$ (as the player of interest) and $r$ so large that $I = \DSet N$. Since $\call N_i^{(r)} = \emptyset$, \eqref{mainform} gives $\call W_2(m_1^i,m_2^i) = 0$. This also implies that, in the proof of \Cref{thmOL}, $\nu = (\mathrm{id}_{\R^d},\mathrm{id}_{\R^d})_\sharp m_1$; then, if $H^i$ is strongly convex, \eqref{esttest1} gives $v_1 = v_2$, $m_1$-a.e.
\end{proof}

\subsection{A different perspective on the unimportance in open-loop games}

Another way to quantify the unimportance of distant players in open-loop games is through decay estimates on the oscillation of the ``decoupling functions'' $v^i$ when only coordinates that are distant (in $\Gamma$) from $i$ vary. This point of view is reminiscent of the approach adopted in \cite{CR24,R_IDNS} in the study of closed-loop sparse games.

The key estimate in this direction is a first variation of \Cref{thmOL} making use of our second technical \Cref{lem_recsumvar}. 

\begin{thm} \label{thmOLvar}
Let \Cref{assOL} be in force. Let $v$ be an admissible solution to \eqref{pspde}, and let $X^{*,i}_t \sim m^{*,i}_t$ and $\hat X^{*,i}_t \sim \hat m^{*,i}_t$ be the solutions to \eqref{eq_stateop}--\eqref{eq_opconOL} with initial data $(m_0^k,(m^j_0)_{j \neq k})$ and $(\hat m_0^k,(m^j_0)_{j \neq k})$, respectively. There exist positive constants $\theta^*\in (0,1)$ and $\tilde\gamma^{(r)}>0$ such that, if \eqref{cond_smallness} holds with $I=\DSet N$ then
\begin{equation} \label{mainformvar}
\call W_2(m^{*,0}_T,\hat m^{*,0}_T)^2 + \int_0^T \call W_2(m^{*,0}_\cdot,\hat m^{*,0}_\cdot)^2 \leq C_k \tilde\gamma^{(r)} \call W_2(m_0^k,\hat m_0^k)^2 ,
\end{equation}
where $r = d_\Gamma(0,k)$ and $C_k$ only depends on $\kappa^0,K_g^0,K_f^0,\kappa^k,\norm{Dv^k}_\infty$. In particular, if there is $\frk n \geq 1$ such that \eqref{eq_Nhkbdd} holds, then for each $\bar\gamma \in (0,1)$ one can choose $\theta^*$ depending only on $\frk n$ and $\bar\gamma$, and such that
\begin{equation} \label{eq_lemrecvarbound}
\tilde\gamma^{(r)} \lesssim \bar\gamma^r,
\end{equation}
with implied constant depending only on $\frk n$.
\end{thm}

In the form we stated it, \Cref{thmOLvar} also holds when $\DSet N = \N$. This is indeed true, provided that we can apply \Cref{lem_recsumvar} with $h^* = +\infty$. For instance, this is the case if we are under the assumptions of \Cref{cor_red} (cf.~\Cref{rmk_extreclemvar}).

\begin{remark} \label{eq_constwithi}
When the reference index is changed from $0$ to $i$, one obtains the same estimate with $i$ in place of $0$ as the player's index, and thus with $r_i \defeq d_\Gamma(i,k)$ in place of $r=r_0$ and a constant $C_{k,i}$ that only depends on $\kappa^i,K_g^i,K_f^i,\kappa^k,\norm{Dv^k}_\infty$, in place of $C_k$.
\end{remark}

%\begin{remark} \label{rmk_decv1}
%We already commented in \Cref{rmk_gammaexp} that for $\theta^*$ sufficiently small we can think of $\gamma^{(j)}$ as $\eta^j$ for some $\eta \in (0,1)$, and the same is clearly true for $\gamma^{(j)}_i$ with $i \neq 0$ as well. On the other hand, $\tilde\gamma_j \lesssim \tilde\eta^{j-r}$ for some $\tilde\eta \in (0,1)$ by its construction (see \Cref{lem_recsumvar}), and similarly $\tilde\gamma_{j,i} \lesssim \tilde\eta^{j-r_i}$. These bounds give 
%\[
%\sum_{j\geq r} \gamma^{(j)}\tilde\gamma_j \lesssim \eta^{d_\Gamma(i,k)}.
%\]
We consider \Cref{thmOL,thmOLvar} two different perspectives on the unimportance of distant players; indeed, the former shows that ignoring distant players causes a small deviation of the optimal trajectories, while the latter demonstrates that they are weakly influenced by variations of the (initial) states of distant players, and, moreover, it allows us to obtain the following decay estimate on the decoupling field $v$.
%\end{remark}

\begin{cor} \label{cor_Dvdec}
Let \Cref{assOL} be in force and let $v$ be an admissible solution to \eqref{pspde}. Suppose that \eqref{cond_smallness} holds with $I = \DSet N$, and that \eqref{eq_Nhkbdd} holds for some $\frk n \geq 1$. Then, for any $\eta \in (0,1)$ there is $\theta^*\in(0,1)$ (depending only on $\frk n$ and $\eta$) such that 
\begin{equation} \label{eq_vdecest}
\norm{D_k v^i}_\infty \lesssim \eta^{d_\Gamma(i,k)} \qquad \forall\,i,k \in \DSet N,
\end{equation}
where the implied constant only depends on $C_{k,i}$ (as in \Cref{eq_constwithi}), $T$, $\norm{D^2f^i}_\infty$, $\norm{D^2g^i}_\infty$ and $\norm{D^2H^i}_\infty$.
\end{cor}

\subsection{Unimportance of distant players in distributed games}

The general strategy of the proof of \Cref{thmOL} can also be used to show a similar result for games with \emph{distributed} controls. In this case, we consider the dynamics \eqref{eq_stateop} with
\begin{equation} \label{eq_opconD}
\beta^{*,i}_t = - D_p H^i(X^{*,i}_t,Dw^i(t,X^{*,i}_t))
\end{equation}
where $w = (w^i)_{i \in \DSet N} \colon [0,T] \times \R^d \to \R^N$ solves the following system of {Hamilton--Jacobi--Bellman} and {Fokker--Planck--Kolmogorov} (FPK) equations:
\begin{equation} \label{dispde}
\begin{dcases}
-\de_t w^i - \tfrac12 \tr(\Sigma^i D^2 w^i) + H^i(x,Dw^i) = \int_{(\R^d)^{N-1}} \!\! f^i(t,x,\bs y)\,\prod_{j \neq i} m^{*,j}_t(\di y^j) \\
\de_t m^{*,i} - \tfrac12 \tr(\Sigma^i D^2m^{*,i}) - \mathrm{div}(m^{*,i} D_pH^i(x,Dw^i)) = 0 \\[3pt]
w^i(T,\var) = \int_{(\R^d)^{N-1}} \!\! g^i(\cdot,\bs y)\,\,\prod_{j \neq i} m^{*,j}_T(\di y^j), \quad m^{*,i}|_{t=0} = m_0^i \in \Pc_2(\R^d) \\[-8pt]
\end{dcases} \vspace{8pt}
\quad i \in \DSet N.
\end{equation}
The FPK equations are understood in the sense of distributions and are no other than the PDE ``translation'' (via It\^o's formula) of the SDEs \eqref{eq_stateop} with drifts given by \eqref{eq_opconD}---that is, $m^{*,i}_t = \call L(X^{*,i}_t)$. The functions $f^i$ and $g^i$ are understood to be computed at $\bs z \in (\R^d)^N$ with $z^i = x$ and $z^j = y^j$ for $j \neq i$. Moreover, due to Assumption~\ref{assOL}(b), they actually only depend on $y^j$ if $j \sim i$, so
\begin{equation} \label{eq_trueintD}
\int_{(\R^d)^{N-1}} \!\! f^i(t,x,\bs y)\,\prod_{j \neq i} m^{*,j}_t(\di y^j) = \int_{(\R^d)^{n_i}}  \!\! f^i(t,x,\bs y)\,\prod_{j \sim i} m^{*,j}_t(\di y^j),
\end{equation}
and likewise for the integral with $g^i$.

Similarly to the connection we stressed between the former setting and open-loop games, in this case $\bs X^*$ can be thought as the vector of optimal trajectories of the game determined by \eqref{eq_state}--\eqref{eq_gcosts} when admissible controls are the so-called \emph{distributed} ones---that is, each $\alpha^i_t = \alpha^i(t,X^i_t)$ is a feedback function of the respective player's state. Formally, this could be considered an intermediate setting between open-loop and closed-loop strategies---that is, between controls just adapted to the Brownian motions' filtration and feedbacks of all players' states; nevertheless, it is worth noting that the distributed formulation entails by construction the independence of the players' states and it is a special case of the open-loop one that occurs by supposing each $\alpha^i$ to be only adapted to the idiosyncratic noise $B^i$ of its respective player. For this reason, such a setting can often be even more tractable than the open-loop one, hence it can be expected that \Cref{thmOL} extends to distributed equilibria in a rather straightforward manner.

We construct the reduced system (or game) similarly as before; more precisely, the players' states will be given by \eqref{eq_stateopred} with
\[
\hat \beta^{*,i}_t = - D_p H^i(\hat X^{*,i}_t,D\hat w^i(t,\hat X^{*,i}_t))
\]
for $\hat w$ that solves 
\begin{equation} \label{dispdered}
\begin{dcases}
-\de_t \hat w^i - \tfrac12 \tr(\Sigma^i D^2 \hat w^i) + H^i(x,D\hat w^i) = \int_{(\R^d)^{N-1}} \!\! f^i(t,x,\bs y)\,\prod_{j \neq i} \hat m^{*,j}_t(\di y^j) \\
\de_t \hat m^i - \tfrac12 \tr(\Sigma^i D^2 \hat m^i) - \mathrm{div}(\hat m^i D_pH^i(x,D\hat w^i)) = 0 \\[3pt]
\hat w^i(T,\var) = \int_{(\R^d)^{N-1}} \!\! g^i(\cdot,\bs y)\,\,\prod_{j \neq i} \hat m^{*,j}_T(\di y^j), \quad \hat m^i|_{t=0} = \hat m_0^i \\[-8pt]
\end{dcases} \vspace{8pt}
\quad i \in I,
\end{equation}
where 
\begin{equation*}% \label{eq_redmeas}
\hat m^{*,i} \defeq \call L(Z^i) \quad \text{for} \quad i \in \bigcup_{j \in I} \call N_j \setminus I;
\end{equation*}
in other words, $\hat m^{*,i} \in \Pc_2(\R^d)$ is arbitrarily assigned for $i \notin I$, also note that we do not really need to define $\hat m^{*,i}$ for all $i \in \DSet N$ as we have in fact the identity \eqref{eq_trueintD}.

The counterpart of \Cref{thmOL} is as follows. Then, all the considerations that stemmed from that theorem in the open-loop setting will carry over to this distributed one. 

\begin{thm} \label{thmD}
Let \Cref{assOL} be in force. Let $(w,m)$ and $(\hat w,\hat m)$ be admissible solutions to \eqref{dispde} and \eqref{dispdered}, respectively. Let $I$ be as in \eqref{asssat} for some $r \in \N$. There exist constants $\theta^*\in(0,1)$ and $\gamma^{(r)} \geq 0$ such that, if \eqref{cond_smallness} holds, then one has estimate~\eqref{mainform}. In particular, if \eqref{eq_Nhkbdd} holds, then $\theta^*$ can be chosen as in \Cref{thmOL} in order to have \eqref{eq_gammaexpd}.
\end{thm}

\section{Examples} %\label{sec_ex}

Estimate~\eqref{mainform} (as well as \eqref{mainformvar}) along with the needed bound \eqref{cond_smallness} and the precise construction of $\gamma^{(r)}$ (and $\tilde\gamma^{(r)}$) in the proof of \Cref{lem_recsum} (and \Cref{lem_recsumvar}), explicitly quantifies the interplay between the structure of the data of the game and that of the graph that determines which other players each one (directly) interacts with. In particular, the form of the graph comes into play through the constants $N_k^h$ defined in \eqref{eq_Nhkbdd}, which determine the upper threshold \eqref{eq_boundthetaA} for $\theta^*$.

We propose here below some basic examples to highlight how the shape of the underlying graph intervenes in determining the constants in \Cref{thmOL}. We leave to the reader similar considerations concerning the additional constants appearing in \Cref{thmOLvar}.

\begin{example} \label{ex_ex1}
Consider the chain graph described by $\call N_i = \{i-1 \ \mathrm{mod}\, N,\, i+1 \ \mathrm{mod}\, N\}$, and let $g^i=0$ and
\[
f^i(t,\bs x) = \sum_{j \sim i} \phi(x^i-\mu x^j)
\]
for some convex function $\phi$ with $0 < \lambda I_d \leq D^2\phi \leq \Lambda I_d$ and $\mu \in (0,1)$. In this case, Assumption~\ref{assOL}(c) holds with
\[
K_f^i = \lambda, \qquad \ell_f^i = \frac{\mu\Lambda}2,
\]
and we have
\[
N_k^h = \begin{dcases}
1 & h \leq \Bigl\lfloor \frac{N}2 \Bigr\rfloor, \ k \in \bigl\{ \pm(h+1), \pm(h-1) \ \mathrm{mod}\, N \bigr\} \setminus \{0\}
\\
2 & (k,h) = (0,1)
\\
0 & \text{otherwise},
\end{dcases}
\]
so, according to \eqref{eq_defgammah},
\[
\gamma_1 = \frac{\theta}{1-2\theta^2}, \qquad \gamma_h = \frac{\theta}{1-\theta \gamma_{h-1}} \quad \text{for} \ \ 2 \leq h \leq \Bigl\lfloor \frac{N}2\Bigr\rfloor.
\]
Also, $\#\call N^{(r)}_0 = 2$, so it suffices that $\gamma^{(r)} \leq \eta^r$ for any $\eta \in (0,1)$ in order for \Cref{cor_red} to hold on large graphs. As $\tilde\gamma^{(r)} = \prod_{j \in \DSet r} \gamma_j$ according to \eqref{eq_defgamma(r)}, this amounts to asking for condition \eqref{cond_smallness} to hold with $\theta^* < \frac12$; therefore, if $\mu\Lambda \leq \lambda$ then our results hold on arbitrarily long horizons. Of course, costs of the same type can be considered on other graphs, and the needed relation between $\lambda$, $\mu$ and $\Lambda$ will depend on their structure.
\end{example}

Note that for any undirected graph $\Gamma$---that is, if $j \in \call N_i$ if and only if $i \in \call N_j$ for all $i,j\in\DSet N$---$N_k^h = 0$ whenever $k \sim_\ell 0$ for $\ell \notin \{h\pm1,h\}$, with in particular $N_0^h = n_0\1_{h=1}$.\footnote{With an abuse of notation, $\1_{h=1}$ denotes the function $\1_{\{1\}}(h)$; in other words, $\1_{h=1} = \delta_{h1}$, where $\delta_{h1}$ is the Kronecker delta.} This simplifies the recursion \eqref{eq_defgammah} that eventually determines $\gamma^{(r)}$. %In this case, the recursion \eqref{eq_defgammah} takes the form
%\[
%\Bigl( 1 - \theta \sup_{k \sim_h 0} N^h_k - \theta \sup_{k\sim_{h-1}0} N_k^h \gamma_{h-1} \Bigr) \gamma_{h} =  \theta \sup_{k \sim_{h+1} 0} N_k^h.
%\]
If in addition $\call N_i \cap \call N_j \subseteq \{0\}$ for any $i,j \in \call N_0^{(h)}$, $h \in \N$, we have 
\begin{equation} \label{eq_simpg}
0 \neq k \not\sim_{h\pm1} 0 \quad \implies \quad N_k^h = 0,
\end{equation}
or equivalently
\[
0 \neq k \sim_\ell 0 \quad \implies \quad N_k^h = 0 \quad \forall \, h \notin \{ \ell \pm 1\}.
\]
This in general happens if the distance from $0$ never stays constant while moving on the graph---that is, whenever $d_{\Gamma}(0,i) \neq d_{\Gamma}(0,j)$ for all $i \in \DSet N$ and $j \in \call N_i$ or, equivalently, whenever $\smash{\bigcup_{j \in \call N_0^{(h)}} \call N_j \cap \call N_0^{(h)} = \emptyset}$ for all $h \in \N$. Some explicit settings are presented in the next example.

\begin{example} \label{ex_simpg}
In \Cref{ex_ex1} we showed that property \eqref{eq_simpg} holds if $\Gamma$ is a chain. Other situations where that happens are if $\Gamma$ has a tree-like or a lattice structure. In the first case, there is a unique path of minimal length joining $0$ to each other vertex and we have $N_k^{h} \leq \1_{k \sim_{h\pm1} 0}$ for all $k \neq 0$. In the second case, supposing that the lattice is planar, up to choosing a bijection $\varsigma \colon \N \to \Z^2$ we can suppose that players are indexed by $i = (i^1,i^2) \in \Z^2$ and $\call N_i = \{ i \pm (1,0), i \pm (0,1) \}$, and the non-zero values of $N_k^h$ with $k \neq 0$ are
\begin{equation} \label{eq_Nlatt}
N_k^h = \begin{dcases}
1 & \text{if } k \sim_{h+1} 0 \text{ and } k^1k^2 = 0 \\
2 & \text{if } k \sim_{h\pm1} 0 \text{ and } k^1k^2 \neq 0 \\
3 & \text{if } k \sim_{h-1} 0 \text{ and } k^1k^2 = 0 .
\end{dcases}
\end{equation}
This can be generalised to higher-dimensional lattices as well. Also note that from \eqref{eq_Nlatt} we see that $\sum_{h \in \N} N_k^h = 4$, with $4 = n_i$ for all $i$---that is, \eqref{eq_trivbN} holds with equality, also recalling that $\check n_i = n_i$. This is in fact true for any undirected graph, for each $k$ which is connected to $0$---that is, if $k \sim_\ell 0$ for some $\ell \in \N$.
\end{example}

\begin{example} \label{ex_dirg}
Some directed graphs of interest can be obtained from those considered in \Cref{ex_simpg} by fixing an orientation for some (possibly all) edges $\{i,j\}$; according to our formalism, this amounts to taking an undirected graph $\Gamma$ determined by a choice of $\call N_i$ and replacing those neighbourhoods with $\tilde{\call N}_i \subseteq \call N_i$. If, for instance, one considers the planar lattice with
\begin{equation} \label{eq_dirlat1}
\tilde{\call N}_i = \{ i + \epsilon_i^1(1,0), i + \epsilon_i^2(0,1) \}, \qquad \epsilon_i^j =\begin{dcases}
\mathrm{sgn}(i^j) & i^j \neq 0 \\
\pm 1 & i^j = 0,
\end{dcases} 
\end{equation}
then the distance from $0$ is increasing along the graph and we have fewer non-zero $N_k^h$'s than in \eqref{eq_Nlatt}; more precisely, the only non-zero ones for $k\neq0$ are
\[
N_k^h = \begin{dcases}
1 & \text{if } k \sim_{h+1} 0 \text{ and } k^1k^2 = 0 \\
2 & \text{if } k \sim_{h+1} 0 \text{ and } k^1k^2 \neq 0. \\
\end{dcases}
\]
In particular, in this case \eqref{eq_defgammah} simply becomes
\[
\gamma_{h} =  \theta \sup_{k \sim_{h+1} 0} N_k^h = 2\theta \qquad \forall\, h \geq 1,
\]
so $\gamma^{(r)} = \frac12(2\theta)^r$, for any $\theta > 0$, thus producing the same threshold for $\theta^*$ as in \Cref{ex_ex1}.
On the other hand, if we modify 
\begin{equation} \label{eq_dirlat2}
\epsilon_{(-1,i^2)}^1 = - \epsilon_{(-1,i^2)}^2 = 1 \quad \text{when} \quad i^2 > 0, \qquad \text{but} \quad \epsilon_{(-1,h)}^1 = -1
\end{equation}
for some $h > 0$, then we can construct graphs where there are vertices arbitrarily distant from $0$ whose neighbourhood contains vertices (directly) connected to $0$; indeed, in this case $(-1,1) \in \call N_{(0,0)}^{(2h)}$ and $(-1,0) \in \call N_{(-1,1)} \cap \call N_{(0,0)}$. As a consequence, we also have $\smash{\sup_{k \sim 0} N^{2h}_k \geq 1}$, which produces a worse scenario than before when it comes to determining the threshold $\theta^*$ according to \eqref{eq_boundthetaA}. 
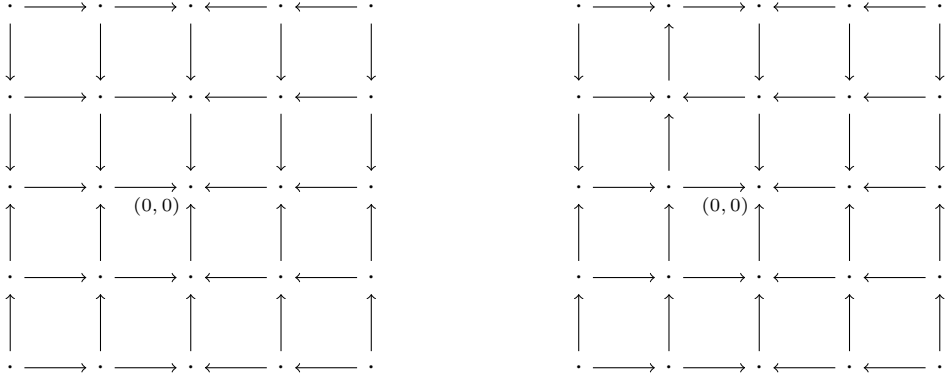
\begin{figure}[htb]
\centering
\begin{minipage}{.45\textwidth}
\centering
\begin{tikzpicture}[x=1.2cm,y=1.2cm]
% nodes
\foreach \i in {-2,-1,0,1,2} \foreach \j  in {-2,-1,0,1,2}
\node (\i\j) at (\i,\j)	 {$\cdot$};
\node[below left] at (0,0) {{\tiny$(0,0)$}};
% edges
\foreach \i in {-2,-1,0,1,2}
\draw[<-] (\i0) edge (\i-1) (\i0) edge (\i1);
\foreach \i in {-2,-1,0,1,2}
\draw[<-]  (0\i) edge (-1\i) (0\i) edge (1\i);
\foreach \i in {-2,-1,0,1,2}
\draw[<-] (\i1) -- (\i2);
\foreach \i in {-2,-1,0,1,2}
\draw[<-] (\i-1) -- (\i-2);
\foreach \i in {-2,-1,0,1,2}
\draw[<-] (1\i) -- (2\i);
\foreach \i in {-2,-1,0,1,2}
\draw[<-] (-1\i) -- (-2\i);
\end{tikzpicture}
\end{minipage}
\begin{minipage}{.45\textwidth}
\centering
\begin{tikzpicture}[x=1.2cm,y=1.2cm]
% nodes
\foreach \i in {-2,-1,0,1,2} \foreach \j  in {-2,-1,0,1,2}
\node (\i\j) at (\i,\j)	 {$\cdot$};
\node[below left] at (0,0) {{\tiny$(0,0)$}};
% edges
\foreach \i in {-2,0,1,2}
\draw[<-] (\i0) edge (\i-1) (\i0) edge (\i1);
\foreach \i in {-2,-1,0,2}
\draw[<-]  (0\i) edge (-1\i) (0\i) edge (1\i);
\foreach \i in {-2,0,1,2}
\draw[<-] (\i1) -- (\i2);
\foreach \i in {-2,-1,0,1,2}
\draw[<-] (\i-1) -- (\i-2);
\foreach \i in {-2,-1,0,1,2}
\draw[<-] (1\i) -- (2\i);
\foreach \i in {-2,-1,0,1,2}
\draw[<-] (-1\i) -- (-2\i);
\draw[->] (-10) edge (-11) (-11) edge (-12) (-1-1) edge (-10) (01) edge (-11) (11) edge (01);
\end{tikzpicture}
\end{minipage}
\caption{The two forms of directed lattice considered in \Cref{ex_dirg}---the former corresponds to the choice \eqref{eq_dirlat1} and the latter is obtained after the change \eqref{eq_dirlat2} with $h=2$.}
\end{figure}
\end{example}

\section{Proofs of the main results} \label{sec_proofs}

We collect in this section the proofs of \Cref{thmOL,thmOLvar,cor_Dvdec,thmD}. As it will be apparent, those of \Cref{thmOLvar,thmD} follows the same lines of that of \Cref{thmOL}: the former is a non-trivial variation, while the latter is a direct adaptation.

\begin{proof}[Proof of \Cref{thmOL}]
Let $\nu \in C^0([0,T];\Pc_2((\R^d)^N \times (\R^d)^I)$ solve
\begin{align} \label{synccouppde}
&\de_t \nu - \tfrac12 \sum_{j \in \DSet N} \tr(\Sigma^j D_{x^jx^j} \nu ) - \tfrac12 \sum_{k \in I} \tr(\Sigma^k D_{\hat x^k\hat x^k} \nu ) - \sum_{k \in I} \tr(\Sigma^k D_{x^k\hat x^k} \nu) 
\notag \\
&- \sum_{j \in \DSet N} \mathrm{div}_{x^j}\bigl(D_p H^j(x^j,v^j(t,\bs x))\nu \bigr) - \sum_{k \in I} \mathrm{div}_{\hat x^k}\bigl(D_p H^k(\hat x^k,\hat v^k(t,\bs{\hat x}))\nu \bigr) = 0
\end{align}
with $\nu(0)$ having as marginals (with respect to the factorisation $(\R^d)^N \times (\R^d)^I = (\R^d \times \R^d)^{I} \times (\R^d)^{\DSet N \setminus I}$) the measures $(\mathrm{id}_{\R^d},\mathrm{id}_{\R^d})_\sharp m_0^i$ for $i \in I$ and $m^i_0$ for $i \notin I$. Testing \eqref{synccouppde} by\footnote{This can be done as long as $v^i$ is an acceptable solution; cf., e.g., \cite{baldi}.}
\begin{equation*}% \label{eq_testsemi}
\bigl( v^i(t,\bs x) - \hat v^i(t,\bs{\hat x}) \bigr) \cdot (x^i - \hat x^i), \qquad i \in I \setminus \call N^{(r-1)}_0,
\end{equation*}
we have
\begin{equation} \label{difftested}
\begin{split}
&\int_0^T \int_{(\R^d)^N \times (\R^d)^I} \!\!\Bigl( - D_x H^i\Bigr\rvert^{(x^i,v^i(t,\bs x))}_{(\hat x^i,\hat v^i(t,\bs{\hat x}))} \cdot (x^i-\hat x^i) + D_p H^i \Bigr\rvert^{(x^i,v^i(t,\bs x))}_{(\hat x^i,\hat v^i(t,\bs{\hat x}))} \cdot \bigl(v^i(t,\bs x) - \hat v^i(t,\bs{\hat x})\bigr) \Bigr) \,\di \nu(t)\di t
\\
&=  - \int_{(\R^d)^{N} \times (\R^d)^{I}} \!\! D_i g^i\bigr|^{\bs x}_{\bs{\hat x}}\cdot (x^i - \hat x^i) \,\di \nu(T) - \int_0^T \int_{(\R^d)^{N} \times (\R^d)^{I}} \!\! D_i f^i(t,\cdot)\bigr|^{\bs x}_{\bs{\hat x}}\cdot (x^i - \hat x^i) \,\di \nu(t)\di t. 
%\\
%&\quad - \bs 1_{i > n} \int_0^T \int_{(\R^d)^{2N}} D_p H^i(0, v^i_{I_\ell}(t,\bs y)) \cdot \bigl(v^i(t,\bs x) - v^i_{I_\ell}(t,\bs y)\bigr) \Bigr) \,\di \nu(t)\di t,
 \end{split}
\end{equation}
%where, for $i > n$, $v^i_{I_\ell} = \hat v_{I_\ell}$ solving
%\begin{equation*}% \label{pspde}
%\begin{dcases}
%-\de_t \hat v_{I_\ell} - \sigma \Delta \hat v_{I_\ell} + D_xH^i(0,\hat v_{I_\ell}) + \sum_{1 \leq j \leq n} D_j \hat v_{I_\ell} D_pH^j(y^j, v^j_{I_\ell}) = D_i f^i\bs y \\[-3pt]
%v^i_{I_\ell}(T,\bs y) = D_ig^i\bs y.
%\end{dcases}
%\end{equation*}
Note that with an abuse of notation we wrote $D_i f^i$ and $D_i g^i$ to be computed both at $\bs x \in (\R^d)^N$ and $\bs{\hat x} \in (\R^d)^{I}$; this is motivated by the fact that, by Assumption~\ref{assOL}(b) and since $i \in I \setminus \call N^{(r-1)}_0$, they are independent of the coordinates $x^j$ with $j \notin I$. Using Assumptions~\eqref{assOL}(a)-(c) and \eqref{difftested} we obtain
\begin{equation} \label{esttest1}
\kappa^i \int_0^T\! \int_{(\R^d)^N \times (\R^d)^I}\!\! \Bigl\lvert D_p H^i\Bigr\rvert^{(x^i,v^i(t,\bs x))}_{(\hat x^i,\hat v^i(t,\bs{\hat x}))} \Bigr\rvert^2 \,\di \nu(t)\di t \leq -K_g^i W^i_T - K_f^i \int_0^T W^i + \ell_g^i  \sum_{j \sim i} W^j_T + \ell_f^i \sum_{j \sim i} \int_0^T W^j,
\end{equation}
where
\[
W_t^j \defeq \int_{\R^d \times \R^d} |x^j-\hat x^j|^2 \,\pi^j_\sharp \nu_t(\di x^j,\di \hat x^j), \qquad j \in I,
\]
$\pi^j$ being the canonical projection of $(\R^d \times \R^d)^{I}$ onto the copy of $\R^d \times \R^d$ indexed by $j$.
On the other hand, testing \eqref{synccouppde} by $|x^i-\hat x^i|^2$, we have
\begin{equation*} %\label{testxiyi}
W_t^i = - 2 \int_0^t \int_{(\R^d)^N \times (\R^d)^I} \!\! D_p H^i\Bigr\rvert^{(x^i,v^i(t,\bs x))}_{(\hat x^i,\hat v^i(t,\bs{\hat x}))} \cdot (x^i-\hat x^i) \,\di \nu(t)\di t
%\\
%&\quad - \bs 1_{i > n} \int_0^t \int_{(\R^d)^{2N}} D_p H^i(0, v^i_{I_\ell}(t,\bs y)) \cdot (x^i-y^i) \,\di \nu(t)\di t,
\end{equation*}
and by H\"older's inequality
\begin{equation} \label{esttest2}
\sup_{[0,T]} W^i \vee \fint_0^T W^i \leq 4T \int_0^T \int_{(\R^d)^N \times (\R^d)^I} \Bigl\lvert D_p H^i\Bigr\rvert^{(x^i,v^i(t,\bs x))}_{(\hat x^i,\hat v^i(t,\bs{\hat x}))} \Bigr\rvert^2 \,\di \nu(t)\di t.
\end{equation}
Combining \eqref{esttest1} and \eqref{esttest2} we deduce that
%\begin{equation} \label{estintd1}
%\bigl( \gamma + K_g T \bigr) W^i_T + K_f T^2 \fint_0^T W^i \leq \ell_g T \sum_{j \sim i} W^j_T + \ell_f T^2 \sum_{j \sim i} \fint_0^T W^j
%\end{equation}
%as well as
%\begin{equation} \label{estintd2}
%K_g T W^i_T + \bigl( \gamma + K_f T^2 \bigr) \fint_0^T W^i \leq  \ell_g T \sum_{j \sim i} W^j_T + \ell_f T^2 \sum_{j \sim i}  \fint_0^T W^j,
%\end{equation}
%so, summing \eqref{estintd1} and \eqref{estintd2},
\begin{equation} \label{estintd3}
\Bigl( \frac{\kappa^i}8 + K_g^i T \Bigr) W^i_T + \Bigl( \frac{\kappa^i}8 + K_f^i T^2 \Bigr) \fint_0^T W^i \leq  \ell_g^i T \sum_{j \sim i} W^j_T + \ell_f^i T^2 \sum_{j \sim i}  \fint_0^T W^j.
\end{equation}
Note now that the above estimates extends to $i \in \call N^{(r-1)}_0$ by letting
\[
W^j_t \defeq 2 \bigl( \frk m_2(m^{*,j}_t) + \frk m_2(Z^j_t) \bigr) \qquad \text{if} \quad j \in \call N^{(r)}_0;
\]
therefore, by \eqref{estintd3} and \Cref{lem_recsum}, there is $\theta^* \in (0,1)$ such that under condition \eqref{cond_smallness} we have
\[
\Bigl( \frac{\kappa^0}8 + K_g^0 T \Bigr) W^0_T + \Bigl( \frac{\kappa^0}8 + K_f^0 T^2 \Bigr) \fint_0^T W^0 \leq \gamma^{(r)} \sum_{j \sim_r 0} \Bigl( \Bigl( \frac{\kappa^j}8 + K_g^j T \Bigr) W^j_T + \Bigl( \frac{\kappa^j}8 + K_f^j T^2 \Bigr) \fint_0^T W^j \Bigr),
\]
from which \eqref{mainform} follows.
\end{proof}

\begin{proof}[Proof of \Cref{thmOLvar}]
Let $\nu$ be as in the proof of \Cref{thmOL}, but with $I = \DSet N$ and $\pi^k_\sharp \nu(0)$ being any coupling of $(m_0^k,\hat m_0^k)$. For any $i \neq k$, the argument used in the proof of \Cref{thmOL} gives \eqref{estintd3}, while if $i = k$ we have
\begin{align*}
&\kappa^k \int_0^T\! \int_{(\R^d)^N \times (\R^d)^N}\!\! \Bigl\lvert D_p H^k\Bigr\rvert^{(x^k,v^k(t,\bs x))}_{(\hat x^k,\hat v^k(t,\bs{\hat x}))} \Bigr\rvert^2 \,\di \nu(t)\di t 
\\[-3pt]
&\leq \norm{Dv^k}_\infty W^k_0 -K_g^k W^k_T - K_f^k \int_0^T W^k + \ell_g^k  \sum_{j \sim k} W^j_T + \ell_f^k \sum_{j \sim k} \int_0^T W^j
\end{align*}
as well as
\[
W_t^k = W_0^k - 2 \int_0^t \int_{(\R^d)^N \times (\R^d)^N} \!\! D_p H^k\Bigr\rvert^{(x^k,v^k(t,\bs x))}_{(\hat x^k,\hat v^k(t,\bs{\hat x}))} \cdot (x^k-\hat x^k) \,\di \nu(t)\di t,
\]
whence
\[
\sup_{[0,T]} W^k \vee \fint_0^T W^k \leq 2W^k_0 + 8T \int_0^T \int_{(\R^d)^N \times (\R^d)^N} \Bigl\lvert D_p H^i\Bigr\rvert^{(x^i,v^i(t,\bs x))}_{(\hat x^i,\hat v^i(t,\bs{\hat x}))} \Bigr\rvert^2 \,\di \nu(t)\di t
\]
and then
\begin{align*}% \label{eq_recestdx}
\Bigl( \frac{\kappa^k}{16} + K_g^k T \Bigr) W^k_T + \Bigl( \frac{\kappa^i}{16} + K_f^k T^2 \Bigr) \fint_0^T W^k 
\leq (2\kappa^k+T \norm{Dv^k}_\infty) W_0^k + \ell_g^k T \sum_{j \sim k} W^j_T + \ell_f^k T^2 \sum_{j \sim k}  \fint_0^T W^j.
\end{align*}
We conclude by invoking \Cref{lem_recsumvar}.
\end{proof}

\begin{proof}[Proof of \Cref{cor_Dvdec}]
For the sake of simplicity, we only give the proof for the case when $D_x H^0 = 0$; for the general one one proceeds similarly, also using the Lipschitz continuity of $D_xH^i$ and Gr\"{o}nwall's lemma. Let $X^{*,i}_t \sim m^{*,i}_t$ and $\hat X^{*,i}_t \sim \hat m^{*,i}_t$ be the solutions to \eqref{eq_stateop}--\eqref{eq_opconOL} with initial data $(\delta_{x^j})_{j \in \DSet N}$ and $(\delta_{\hat x^k},(\delta_{x^j})_{j \neq k})$, respectively. Let $\nu$ be as in the proof of \Cref{thmOLvar} (with $m_0^j = \delta_{x^j}, \hat m_0^k = \delta_{\hat x^k}$). Integrating $v^0(t,\bs y) - v^0(t,\bs{\hat y})$ against $\nu(\di \bs y,\di \bs{\hat y})\di t$ we obtain
\[ \begin{split}
v^0(0,(x^j)_{j\neq k},\cdot)\bigr|^{x^k}_{\hat x^k} &= \int_0^T \int_{(\R^d)^N \times (\R^d)^N} \!\! D_0f^0(t,\cdot)\bigr|^{\bs y}_{\bs{\hat y}}\, \nu_t(\di\bs y,\di\bs{\hat y})\di t + \int_{(\R^d)^N \times (\R^d)^N} \!\! D_0g^0\bigr|^{\bs y}_{\bs{\hat y}} \,\nu_T(\di\bs y,\di\bs{\hat y})
\\
&\leq  \sum_{j \in \call N_0 \cup \{0\}} \! \! \biggl( \norm{D^2f^0}_\infty \sqrt{T} \biggl( \,\int_0^T \call W_2(m^{*,j}_t,\hat m^{*,j}_t)^2\,\di t \biggr)^{\frac12} \! + \norm{D^2g^0}_\infty \call W_2(m^{*,j}_T,\hat m^{*,j}_T) \biggr);
\end{split}
\]
then \eqref{eq_vdecest} follows from \eqref{mainformvar}. Note that to have the estimate at general time $t_0$ it suffices to change the initial time from $t=0$ to $t=t_0$, and to have it for all $i \in \DSet N$ it suffices to consider player $i$ as the player ``of interest'' in place of player $0$.
\end{proof}

\begin{proof}[Proof of \Cref{thmD}]
For $i \in I$, let $\nu^i \in C^0([0,T];\Pc_2(\R^d \times \R^d)$ solve
\begin{align} \label{synccouppdeD}
\begin{dcases}
\de_t \nu^i - \tfrac12 \tr(\Sigma^i D^2_{x} \nu^i ) - \tfrac12 \tr(\Sigma^i D_{\hat x}^2 \nu^i ) - \tr(\Sigma^i D_{x\hat x} \nu^i) \\
\ {}- \mathrm{div}_{x}\bigl(D_p H^i(x,Dw^i(t,x))\nu^i \bigr) - \mathrm{div}_{\hat x}\bigl(D_p H^i(\hat x,D\hat w^i(t,{\hat x}))\nu^i \bigr) = 0
\\
\nu^i|_{t=0} = (\mathrm{id}_{\R^d},\mathrm{id}_{\R^d})_\sharp m_0^i.
\end{dcases}
\end{align}
Testing \eqref{synccouppdeD} by $\bigl( Dw^i(t, x) - D\hat w^i(t,{\hat x}) \bigr) \cdot (x - \hat x)$, we have
\begin{equation} \label{difftestedD}
\begin{split}
&\int_0^T \int_{\R^d \times \R^d} \!\!\Bigl( - D_x H^i\Bigr\rvert^{(x,Dw^i(t, x))}_{(\hat x, D\hat w^i(t,{\hat x}))} \cdot (x-\hat x) + D_p H^i \Bigr\rvert^{(x,Dw^i(t, x))}_{(\hat x,D \hat w^i(t,{\hat x}))} \cdot \bigl(Dw^i(t, x) - D\hat w^i(t,{\hat x})\bigr) \Bigr) \,\di \nu^i_t \di t
\\
&=  - \int_{(\R^d)^{N} \times (\R^d)^{N}} \!\! D_i g^i\bigr|^{(x,\bs y)}_{({\hat x},\bs{\hat y}) }\cdot (x - \hat x) \, \nu^i_T(\di x,\di \hat x)\, \varpi^i_T(\di \bs y,\di \bs{\hat y}) \\
&\quad - \int_0^T  \int_{(\R^d)^{N} \times (\R^d)^{N}} \!\! D_i f^i(t,\cdot) \bigr|^{(x,\bs y)}_{({\hat x},\bs{\hat y}) }\cdot (x - \hat x) \, \nu^i_t(\di x,\di \hat x)\, \varpi^i_t(\di \bs y,\di \bs{\hat y})\,\di t,
 \end{split}
\end{equation}
where $\varpi^i_t$ is any coupling of $\prod_{j \sim i} m^{*,j}_t$ and $\prod_{j \sim i} \hat m^{*,j}_t$. In particular, we can choose
\[
\varpi^i_t = \prod_{j \sim i:\, j \in I} \nu^j_t \otimes \prod_{j \sim i:\, j \notin I} \bigl( m^{*,j}_t \otimes \call L(Z^j_t) \bigr).
\]
Using Assumptions~\eqref{assOL}(a)-(c) and \eqref{difftestedD} we obtain
\begin{equation*}% \label{esttest1D}
\kappa^i \int_0^T\! \int_{\R^d \times \R^d}\!\! \Bigl\lvert D_p H^i\Bigr\rvert^{(x,Dw^i(t, x))}_{(\hat x,D \hat w^i(t,{\hat x}))} \Bigr\rvert^2 \,\di \nu_t^i \di t \leq -K_g^i W^i_T - K_f^i \int_0^T W^i + \ell_g^i  \sum_{j \sim i} W^j_T + \ell_f^i \sum_{j \sim i} \int_0^T W^j,
\end{equation*}
where
\[
W^j \defeq \begin{cases}
\int_{\R^d \times \R^d} |x-\hat x|^2 \,\nu^j(\di x,\di \hat x) & j \in I \\
2 \bigl( \frk m_2(m^{*,j}) + \frk m_2(Z^j) \bigr) & j \in \bigcup_{i \in I} \call N_i \setminus I.
\end{cases}
\]
On the other hand, testing \eqref{synccouppdeD} by $|x-\hat x|^2$ we have
\begin{equation*} %\label{testxiyiD}
W_t^i = - 2 \int_0^t \int_{\R^d \times \R^d} \!\! D_p H^i\Bigr\rvert^{(x,Dw^i(t x))}_{(\hat x, D \hat w^i(t,{\hat x}))} \cdot (x-\hat x) \,\di \nu^i_t\di t
%\\
%&\quad - \bs 1_{i > n} \int_0^t \int_{(\R^d)^{2N}} D_p H^i(0, v^i_{I_\ell}(t,\bs y)) \cdot (x^i-y^i) \,\di \nu(t)\di t,
\end{equation*}
and by H\"older's inequality
\begin{equation*}% \label{esttest2D}
\sup_{[0,T]} W^i \vee \fint_0^T W^i \leq 4T \int_0^T \int_{\R^d \times \R^d} \Bigl\lvert D_p H^i\Bigr\rvert^{(x,Dw^i(t, x))}_{(\hat x,D \hat w^i(t,{\hat x}))} \Bigr\rvert^2 \,\di \nu_t^i\di t.
\end{equation*}
Then the proof proceeds as that of \Cref{thmOL}.
\end{proof}

\section{Two decay estimates for recursive inequalities on graphs} \label{sec_tl}

This last section contains the two technical results used in the proofs above in order to transform the recursive inequalities found into the final estimates \eqref{mainform} and \eqref{mainformvar}, respectively. They are general inequalities for sequences indexed by the vertices of a given graph of $N$ vertices.

\begin{lem} \label{lem_recsum}
Let $A = (A^i)_{i \in \N},B = (B^i)_{i \in \N}$ be non-negative sequences and let $r \in \N$. Suppose there are positive constants $\alpha_i,\beta_i,\alpha'_i,\beta'_i$ such that
\[
\alpha_i A^i+ \beta_i B^i \leq \alpha'_i \sum_{j \sim i} A^j + \beta'_i \sum_{j \sim i} B^j \qquad \forall \, i \sim_k 0,\ k \in \DSet r.
\]
There are $\theta^*\in(0,1)$ and $\gamma^{(r)} \geq 0$ such that, if
\[
\gamma \defeq \sup_{i \sim_k 0,\, k \in \DSet r} \Bigl( \, \frac{\alpha'_i}{\displaystyle\inf_{j \sim i} \alpha_j} \vee \frac{\beta'_i}{\displaystyle\inf_{j \sim i} \beta_j} \Bigr) \leq \theta^*,
\]
then
\begin{equation} \label{eq_lemrec}
\alpha_0 A^0 + \beta_0 B^0 \leq \gamma^{(r)} \sum_{i \sim_r 0} \bigl( \alpha_i A^i + \beta_i B^i \bigr),
\end{equation}
In particular, if there is $\frk n \geq 1$ such that \eqref{eq_Nhkbdd} holds,
then for each $\bar\gamma \in (0,1)$ one can choose $\theta^*$ depending only on $\frk n$ and $\bar\gamma$, and such that \eqref{eq_gammaexpd} holds.%More explicitly, $\gamma^{(r)}$ is defined as in \eqref{eq_boundgammar}, where the $\gamma_h$'s are given by \eqref{eq_boundgamma} with $\theta = \gamma$, and $\theta^*$ satisfies \eqref{eq_boundtheta}. 
\end{lem}

\begin{proof}
Write $C^i \defeq \alpha^i A^i + \beta^i B^i$ and note that
\begin{equation} \label{recineqC}
C^i\leq \gamma \sum_{j \sim i} C^j \qquad \forall\, i \sim_k 0, \ k \in \DSet r.
\end{equation}
We prove by induction that for each $h\geq 0$ there is $\gamma_{h} > 0$ such that
\begin{equation} \label{recineqansatz}
\sum_{i \sim_{h} 0} C^i \leq \gamma_{h} \sum_{i \sim_{h+1} 0} C^i ,
\end{equation}
which we know is true for $h=0$. Summing \eqref{recineqC} over $i \sim_h 0$ to get, for $h \in \DSet r \setminus \{0\}$,
\begin{equation} \label{eq_defgammah} \begin{split}
\sum_{i \sim_h 0} C^i \leq \gamma \sum_{\ell=0}^{h+1} \sum_{j \sim_\ell 0} N_j^h C^j \leq \gamma \sum_{j \sim_{h+1} 0} N_{j}^h C^j + \gamma \sum_{\ell = 0}^{h} \sup_{k \sim_\ell 0} N_k^h \prod_{j=\ell}^{h-1} \gamma_{j} \sum_{i \sim_h 0} C^i;
\end{split}
\end{equation}
then it suffices to let $\gamma_0 \defeq \gamma$ and
\[
\gamma_{h} \defeq \frac{\displaystyle \gamma \sup_{k \sim_{h+1} 0} N_k^h}{\displaystyle 1 - \gamma \sum_{i = 0}^{h} \sup_{k \sim_i 0} N_k^h \prod_{j=i}^{h-1} \gamma_{j}} \qquad \text{for} \quad h \in \DSet r \setminus \{0\}
\]
to conclude that \eqref{recineqansatz} in fact holds, as long as $\gamma_h > 0$. This condition in turn holds provided that $\gamma \leq \theta^*$ with $\theta^*$ small enough---depending on $\gamma_j$ only for $j<h$; in particular, 
\begin{equation} \label{eq_boundthetaA}
\theta^* < \inf_{h \in \DSet{r}} \Bigl(\,  \sum_{i = 0}^{h} \sup_{k \sim_i 0} N_k^h \prod_{j=i}^{h-1} \gamma_{j} \Bigr)^{-1},
\end{equation}
we can iterate estimate~\eqref{recineqansatz} to obtain \eqref{eq_lemrec} with
\begin{equation} \label{eq_defgamma(r)}
\gamma^{(r)} \defeq \prod_{j \in \DSet r} \gamma_j.
\end{equation}
Note now that if there exists $\frk n \geq 1$ such that $N^h_k \leq \frk n$ for all $h,k \in \N$, then for any $K > \frk n$ and $\gamma$ small enough (only depending on $K$) we have $\gamma_h \leq K\gamma(1-K\gamma)^{-1}$ for all $h \geq 1$. Indeed, $\gamma_1 \leq {\gamma\frk n}/({1-\gamma\frk n - \gamma^2})$ and by induction
\[
\gamma_{h} \leq \frac{\gamma \frk n}{1- \gamma \frk n \sum_{i=0}^h \gamma^{h-i}(1-K\gamma)^{i-h} } \leq \frac{\gamma\frk n}{1-\gamma \frk n \frac{1-K\gamma}{1-2K\gamma}} = \frac{\frk n(1-3K\gamma + 2K^2\gamma^2)}{1-2K\gamma - \gamma \frk n + K\gamma^2 \frk n}\, \gamma (1-K\gamma)^{-1},
\]
where $\frac{\frk n(1-3K\gamma + 2K^2\gamma^2)}{1-2K\gamma - \gamma \frk n + K\gamma^2 \frk n} \leq K$ if $\gamma$ is sufficiently small with respect to $K$. In particular, if $\gamma < (2K)^{-1}$ then $\gamma_h < 1$ uniformly for all $h$; the last part of the thesis follows with $\gamma = \bar\gamma/(K(1+\bar\gamma))$, for example with $K = 2\frk n$.
\end{proof} 

\begin{lem} \label{lem_recsumvar}
Let $A = (A^i)_{i \in \N},B = (B^i)_{i \in \N}$ be non-negative sequences. Let $k \in \DSet N$ such that $k \sim_r 0$.\footnote{That is, $k$ is connected to $0$ and $r \in \N$ is the unique number such that $k \in \call N^{(r)}_0$.} Suppose that there are positive constants $\alpha_i,\beta_i,\alpha'_i,\beta'_i,E^k$ such that 
\[
\alpha_i A^i+ \beta_i B^i \leq E^k \1_{i=k} + \alpha'_i \sum_{j \sim i} A^j + \beta'_i \sum_{j \sim i} B^j \qquad \forall\, i \in \DSet N
\]
There are constants $\theta^*\in (0,1)$ and $\tilde\gamma^{(r)}>0$ such that, if
\[
\gamma \defeq \sup_{i \in \DSet N} \Bigl( \, \frac{\alpha'_i}{\displaystyle\inf_{j \sim i} \alpha_j} \vee \frac{\beta'_i}{\displaystyle\inf_{j \sim i} \beta_j} \Bigr) \leq \theta^*,
\]
then
\begin{equation} \label{eq_lemrecvar}
\alpha_0 A^0 + \beta_0 B^0 \leq \tilde\gamma^{(r)} E^k.
\end{equation}
In particular, if there is $\frk n \geq 1$ such that \eqref{eq_Nhkbdd} holds, then for each $\bar\gamma \in (0,1)$ one can choose $\theta^*$ depending only on $\frk n$ and $\bar\gamma$, and such that \eqref{eq_lemrecvarbound} holds, with implied constant depending only on $\frk n$.
\end{lem}

%The sum in \eqref{eq_lemrecvar} is actually finite since in the proof we will be able to let $\tilde\gamma^{(j)} = 0$ if $j > h^* \defeq \min\{h \in \N:\ \call N^{(h)}_0 = \emptyset\}$. Note that $h^*$ is finite as so is $N$; nevertheless, the uniform bound \eqref{eq_lemrecvarbound} suggests that there are regimes where \Cref{lem_recsumvar} can be extended to graphs with countably many vertices. This will be further discussed after the proof---see \Cref{rmk_extreclemvar}.

\begin{proof}%[Proof of \Cref{lem_recsumvar}]
Write $C^i \defeq \alpha_i A^i + \beta_i B^i$ and note that
\begin{equation} \label{recineqCvar}
C^i\leq  E^k \delta_{ik} + \gamma \sum_{j \sim i} C^j \qquad \forall\,i \in \DSet N.
\end{equation}
We prove by induction that for each $h\geq 0$ there are $\tilde\gamma_h,\gamma_{h} > 0$ such that
\begin{equation} \label{recineqansatzvar}
\sum_{i \sim_{h} 0} C^i \leq \tilde\gamma_h E^k \1_{h \geq r} + \gamma_{h} \sum_{i \sim_{h+1} 0} C^i ,
\end{equation}
which we know is true for $h < r$ by \Cref{lem_recsum}. Summing \eqref{recineqCvar} over $i \sim_r 0$ we get 
\[
\sum_{i \sim_r 0} C^i \leq E^k + \gamma \sum_{j \sim_{r+1} 0} N_{j}^r C^j + \gamma \sum_{\ell = 0}^{r} \sup_{k \sim_\ell 0} N_k^r \prod_{j=\ell}^{r-1} \gamma_{j} \sum_{i \sim_r 0} C^i,
\]
so \eqref{recineqansatzvar} holds for $h=r$ by defining
\begin{equation} \label{eq_deftgammar}
\tilde \gamma_r \defeq \frac{\gamma_r}{\gamma \sup_{j \sim_{r+1}0} N^r_j}.
\end{equation}
Let now $h>r$ and assume that \eqref{recineqansatzvar} holds for any $h' \in \{0,\dots,h-1\}$ in place of $h$. It is easy to prove by induction that
\[
\sum_{i \sim_{\ell} 0} C^i \leq \sum_{i=\ell}^{h-1} \tilde\gamma_i \prod_{j=\ell}^{i-1} \gamma_j E^k + \prod_{j=\ell}^{h-1} \gamma_j \sum_{i \sim_h 0} C^i \qquad \forall\, \ell \in \{r,\dots,h-1\};
\]
therefore, summing \eqref{recineqCvar} over $i \sim_h 0$ we have
\[ \begin{split}
\sum_{i \sim_h 0} C^i &\leq \gamma \sum_{\ell=0}^{h+1} \sum_{j \sim_\ell 0} N^h_j C^j 
\\
&\leq \gamma \sum_{j\sim_{h+1} 0} N^h_j C^j + \gamma \sum_{\ell = 0}^r \sup_{k \sim_\ell 0} N^h_k \prod_{j=\ell}^{r-1} \gamma_j \sum_{i \sim_r 0} C^i + \gamma \sum_{\ell=r+1}^{h} \sup_{k \sim_\ell 0} N^h_k \sum_{i \sim_\ell 0} C^i
%\\
%&\leq \gamma \sum_{j \sim_{h+1} 0} N_{j}^h C^j + \gamma \Bigl( \sum_{\ell = 0}^{r} \sup_{k \sim_\ell 0} N^h_k \sum_{i=r}^{h-1} \tilde\gamma_i \prod_{j=\ell}^{i-1} \gamma_{j} + \sum_{\ell = r+1}^{h-1} \sup_{k \sim_\ell 0} N^h_k \sum_{i=\ell}^{h-1} \tilde\gamma_i \prod_{j=\ell}^{i-1} \gamma_{j} \Bigr) E^k + \gamma \sum_{\ell = 0}^{h} \sup_{k \sim_\ell 0} N_k^h \prod_{j=\ell}^{h-1} \gamma_{j} \sum_{i \sim_h 0} C^i
\\
&\leq \gamma \sum_{j \sim_{h+1} 0} N_{j}^h C^j + \gamma \sum_{\ell = 0}^{h-1}  \sum_{i = r \vee \ell}^{h-1} \sup_{k \sim_\ell 0} N_k^h \tilde\gamma_i \prod_{j=\ell}^{i-1} \gamma_j E^k + \gamma \sum_{\ell = 0}^{h} \sup_{k \sim_\ell 0} N_k^h \prod_{j=\ell}^{h-1} \gamma_{j} \sum_{i \sim_h 0} C^i.
\end{split}
\]
This proves that \eqref{recineqansatzvar} holds if $\theta^*$ is small enough and we recursively define
\begin{equation} \label{eq_deftgammar+}
\tilde\gamma_j \defeq \frac{\displaystyle \gamma \sum_{\ell = 0}^{j-1}  \sum_{i = r \vee \ell}^{j-1} \sup_{k \sim_\ell 0} N_k^j \tilde\gamma_i \prod_{h=\ell}^{i-1} \gamma_h}{\displaystyle 1 - \gamma \sum_{i = 0}^{j} \sup_{k \sim_i 0} N_k^h \prod_{h=i}^{j-1} \gamma_{h}} \qquad \text{for} \quad j > r.
\end{equation}
It suffices now to iterate \eqref{recineqansatzvar} for $h=0,\dots,h^* \defeq \min\{h \in \N:\ \call N^{(h)}_0 = \emptyset\}$ to obtain \eqref{eq_lemrecvar} with $\tilde\gamma^{(r)} = \sum_{j = r}^{h^*} \gamma^{(j)} \tilde\gamma_j$ (with $\gamma^{(j)}$ defined as in \eqref{eq_defgamma(r)}). Arguing now as in the proof of \Cref{lem_recsum} we can choose $\theta^* = \theta^*(\bar\gamma,\frk n)$ such that, if $\gamma \leq \theta^*$, then
\begin{equation} \label{eq_provedlemprec}
\sup_{h \in \N} \frac{\gamma \frk n}{1- \gamma \frk n \sum_{i=0}^h \gamma^{h-i}(1-K\gamma)^{i-h} } \leq \bar\gamma,
\end{equation}
which in particular implies that
\begin{equation} \label{eq_boundgammaj}
\gamma^{(j)} \leq \bar\gamma^j \qquad \forall\, j \geq 1.
\end{equation}
Then by \eqref{eq_deftgammar+} and \eqref{eq_provedlemprec} we obtain that
\begin{equation} \label{eq_deftgammar+'}
\tilde\gamma_j \leq \bar\gamma \sum_{\ell=0}^{j-1} \sum_{i=r \vee \ell}^{j-1} \tilde \gamma_i \prod_{h=\ell}^{i-1} \gamma_h \leq \sum_{\ell=0}^{j-1} \sum_{i=r \vee \ell}^{j-1} \tilde \gamma_i \bar \gamma^{i-\ell+1} \qquad \text{for} \quad j > r.
\end{equation}
We prove now that, if $\bar\gamma \leq \frac{\sqrt{17}-1}8$, then
\begin{equation} \label{eq_boundtgamma}
\tilde\gamma_j \leq K' (2{\bar\gamma})^{r-j}
\end{equation}
with $K' > 2C/\frk n$. By \eqref{eq_deftgammar}, $\tilde\gamma_r \leq \bar\gamma/(\gamma\frk n) = (C/\frk n)(1+\bar\gamma) \leq K'$ and then, by induction, \eqref{eq_deftgammar+'} gives, for $j > r$,
\[ \begin{split}
\tilde \gamma_j &\leq K' \bar\gamma (2{\bar\gamma})^r \biggl( \sum_{\ell=0}^{r} \bar\gamma^{-\ell} \sum_{i=r}^{j-1} 2^{-i} + \sum_{\ell=r+1}^{j-1} \bar\gamma^{-\ell} \sum_{i=\ell}^{j-1} 2^{-i} \biggr) 
\\
&\leq  K' \bar\gamma (2{\bar\gamma})^r \frac{2}{1-\bar\gamma} \bigl( (2\bar\gamma)^{-r} + (2\bar\gamma)^{1-j} \bigr) \leq \frac{4\bar\gamma^2}{1-\bar\gamma} K'  (2{\bar\gamma})^{r-j},
\end{split}
\]
which implies \eqref{eq_boundtgamma} when $\bar\gamma \leq \frac{\sqrt{17}-1}8$. We conclude the proof by noting that the bounds \eqref{eq_boundgammaj} and \eqref{eq_boundtgamma} give
\[
\tilde\gamma^{(r)} = \sum_{j = r}^{h^*} \gamma^{(j)}\tilde\gamma_j \leq K' \bar\gamma^r \sum_{j \geq r} 2^{r-j} = 2K'\bar\gamma^r. \qedhere
\]
Clearly, for $\bar\gamma \in (\frac{\sqrt{17}-1}8,1)$ it suffices to consider $\theta^* = \theta^*(\frac{\sqrt{17}-1}8,\frk n)$ to obtain a shaper bound than \eqref{eq_lemrecvarbound}.
\end{proof} 

%\begin{remark}
%Note that following the above computations, given $\epsilon > 0$ one has more in general that $\tilde\gamma_j \leq K'((1+\epsilon)\bar\gamma)^{r-j}$ provided that $2(1+\epsilon)^2 \bar\gamma^2 \leq (1-\bar\gamma)\epsilon$. Explicit computations show that the resulting threshold for $\bar\gamma$ is optimised exactly for $\epsilon = 1$, which thus allows for a larger $\theta^*$.
%\end{remark}

\begin{remark} \label{rmk_extreclemvar}
In the proof we used that $h^*$ is finite; nevertheless, \Cref{lem_recsumvar} extends to $h^* = +\infty$ as well, by continuing to iterate the inequality~\eqref{recineqansatzvar}, provided that the term $\gamma^{(h)} \sum_{i \sim_h 0} C^i$ vanishes as $h \to \infty$; this happens, for instance, if the $C^i$'s are uniformly bounded and $\gamma^{(h)} \#\call N_0^{(h)}$ vanishes as $h \to \infty$, which is true if $\theta^*$ is small enough and the graph has uniformly bounded in-degrees. On the other hand, it is clear that \Cref{lem_recsum} holds as it is for graphs with countably many vertices as well.
\end{remark}

\end{document}